\documentclass{amsart}
\usepackage{}
\usepackage{cases}
\usepackage{amsfonts}
\usepackage{amssymb}
\usepackage{mathrsfs}
\usepackage{amsmath}
\usepackage{color}
\usepackage{extarrows}

\newtheorem{theorem}{Theorem}[section]

\newtheorem{lemma}[theorem]{Lemma}
\newtheorem{proposition}[theorem]{Proposition}
\newtheorem{remark}{Remark}
\numberwithin{equation}{section}

\newcommand{\trace}{\mathop{\mathrm{tr}}}
\newcommand{\hyperg}[4]{\: _2\! F_1\! \left[\!\! \begin{array}{c} #1,\, #2 \\ #3 \end{array}\!\! ;\,  #4 \right]}

\newcommand{\sphere}{S}
\newcommand{\ball}{B}

\newcommand{\bbD}{\mathbb{D}}
\newcommand{\bbR}{\mathbb{R}}

\newcommand{\mob}{\mathrm{M\ddot{o}b}}

\begin{document}
\title[Weighted Composition Operators]{Weighted Composition Operators Acting on Harmonic Hardy Spaces}

\thanks{The second author was supported by the National Natural Science
Foundation of China grants 11571333, 11471301; The third author was supported by the
National Natural Science Foundation of China grants 11471111; the fourth author was supported by
Natural Science Foundation of Zhejiang province grant (No. LQ13A010005), the Scientific
Research and Teachers project of Huzhou University (No. RP21028) and partially by the
National Natural Science Foundation of China grant 11571105.}

\date{}

\author[P. Hu]{Pengyan Hu}
\address{Colledge of Mathematics and Statistics,
         Shenzhen University, Shenzhen, Guangdong 518060,
         People's Republic of China.}

\email{pyhu@szu.edu.cn}

\author[C. Liu]{Congwen Liu}

\address{School of Mathematical Sciences,
University of Science and Technology of China,
Hefei, Anhui 230026,
People's Republic of China.}

\address{Wu Wen-Tsun Key Laboratory of Mathematics,
USTC, Chinese Academy of Sciences, Hefei,
People's Republic of China.}

\email{cwliu@ustc.edu.cn}

\author[T. Liu]{Taishun Liu}
\address{Department of Mathematics, Huzhou Teachers College, Huzhou, Zhejiang 313000,
People¡¯s Republic of China.}

\email{tsliu@zjhu.edu.cn}

\author[L. Zhou]{Lifang Zhou}
\address{Department of Mathematics, Huzhou Teachers College, Huzhou, Zhejiang 313000,
People¡¯s Republic of China.}

\email{lfzhou@zjhu.edu.cn}

\begin{abstract}
Suppose $n\geq 3$ and let $\ball$ be the open unit ball in $\bbR^n$. Let $\varphi: \ball\to \ball$ be a $C^2$ map
whose Jacobian does not change sign, and let $\psi$ be a $C^2$ function on $\ball$.
We characterize bounded weighted composition operators $W_{\varphi,\psi}$ acting on harmonic Hardy spaces
$h^p(\ball)$. In addition, we compute the operator norm of $W_{\varphi,\psi}$ on $h^p(\ball)$ when $\varphi$ is
a M\"obius transformation of $\ball$.
\end{abstract}

\commby{}
%-------------------------------------------------------------------------------------
\maketitle
%-------------------------------------------------------------------------------------
\section{Introduction}

We begin by recalling the notions of composition operators and weighted composition operators
in the classical setting. Let $\bbD$ denote the open unit disk in the complex plane.
Let  $\varphi$ be an analytic self-map of $\bbD$ and let $\psi$ be an analytic function on $\bbD$.
The weighted composition operator $W_{\varphi,\psi}$ is defined on the space of analytic functions on
$\bbD$ by
\begin{equation}\label{eqn:wco}
W_{\varphi,\psi} f := \psi\cdot (f \circ \varphi).
\end{equation}
There are two particularly interesting special cases of such operators: on
one hand, taking $\psi=1$ gives the composition operator $C_{\varphi}$, and on the other,
putting $\varphi=\mathrm{id}$, the identity map of $\bbD$, gives the multiplication operator $M_{\psi}$.

Weighted composition operators arise naturally in many situations.
For example, a classical result due to Forelli \cite{For64} states that all surjective isometries
of the Hardy space $H^p (\mathbb{D})$, $1\leq p<\infty$, $p\neq 2$, are of this form.
They also arise in the description of commutants of analytic Toeplitz operators (see for example \cite{Cow78, Cow80}) and in the
adjoints of (unweighted) composition operators (see for example \cite{Cow88, CG06}).
Of course, they are also of interest in their own right. Recently, there have been an increasing interest
in studying weighted composition operators acting on various spaces of holomorphic functions.
See \cite{CK10, CZ04, CZ07, Gun07, Gun08, Gun11}, to mention only a few.
We refer to \cite{CM95}, \cite{Sha93} or \cite[Chapter 11]{Zhu07} for more about the classical background.

It is clear that, with analytic symbols $\varphi$ and $\psi$, the weighted composition operator $W_{\varphi,\psi}$ is also
well defined on spaces of harmonic functions on $\bbD$. But there is little new to be said.
Just as in the usual Hardy space setting, it follows from Littlewood's subordination principle that
each composition operator $C_{\varphi}$ induces a bounded linear operator on the harmonic Hardy space $h^p(\bbD)$.
If further $\psi$ is bounded on $\bbD$, then $W_{\varphi,\psi}$ is bounded on $h^p(\bbD)$.
For $1<p<\infty$, the compactness problems for composition operators $C_{\varphi}$
on both $h^p(\bbD)$ and $H^p(\bbD)$ are trivially the same, by the Riesz projection theorem.
However, when $p=1$, the Riesz projection theorem no longer operates.
Sarason \cite{Sar90} showed that the compactness of $C_{\varphi}$ on $h^1(\bbD)$ implies
its compactness on $H^2(\bbD)$ (and so on $H^p(\bbD)$ for all $0<p<\infty$).  The reverse implication
was proved by Shapiro and Sundberg \cite{SS90}.  A vector-valued analogue of these
results was recently established by Laitila and Tylli \cite{LT06}. See also \cite{IIO11, IIO14}
for a variant of \eqref{eqn:wco} acting on $h^{\infty}(\bbD)$.

We are concerned in this paper with weighted composition operators of the form \eqref{eqn:wco},
acting on harmonic Hardy spaces $h^p(\ball)$, where $\ball$ is the unit ball in $\bbR^n$
with $n\geq 3$.  The situation turns out to be substantially different in this setting.
Note that $f\circ \varphi$ may not be harmonic even both $f$ and $\varphi$ are harmonic.
Also, it is well known that the product of two harmonic functions is not necessarily a harmonic function.
So, at the beginning, we assume that $\varphi$ and $\psi$ are smooth but not necessarily harmonic or
holomorphic in $\ball$ (the holomorphicity of a function is meaningless in this setting, at least in odd dimensions).
The purpose of this paper is to find necessary and sufficient conditions on $\varphi$ and $\psi$ in order
for $W_{\varphi,\psi}$ to be well defined and bounded on $h^p(\ball)$.

Although in a different spirit, this is partially motivated by the work of Koo and Wang \cite{KW09}, in which the authors studied
the composition operator $C_{\varphi}$ with a smooth symbol $\varphi$, acting from
the weighted harmonic Bergman spaces $b_{\alpha}^p(\ball)$
to $L_{\alpha}^p(\ball)$. They did not include the limit case $\alpha=-1$,
the harmonic Hardy spaces $h^p(\ball)$.

Before stating our main results, we introduce some definitions and notation.

For a fixed positive integer $n\geq 2$, let $\ball = B_{n}$ be the open unit ball
and $\sphere = \partial \ball$ be the unit sphere in $\mathbb{R}^n$. Let $d\sigma$ be the surface area measure
on $\sphere$, normalized so that $\sigma(\sphere)=1$. For $1\leq p<\infty$, the harmonic Hardy space
$h^p(\ball)$ consists of all complex-valued harmonic function $f$ on $\ball$ such that
\[
\|f\|_{h^p}:= \sup_{0\leq r<1} \bigg\{\int\limits_{\sphere} |f(r\zeta)|^p d\sigma(\zeta)\bigg\}^{\frac {1}{p}} < \infty.
\]
Also, let $h^\infty(\ball)$ be the Banach space of complex-valued bounded harmonic function on $\ball$, with norm
\[
\|f\|_{h^\infty}:=\sup_{x\in\ball}|f(x)|.
\]

A M\"obius transformation of $\widehat{\mathbb{R}}^n$ (the one-point compactification of $\mathbb{R}^n$)
is a finite composition of reflections in spheres or hyperplanes. See the next section for the
detailed definitions.

For a differentiable map $\varphi: \Omega\to f(\Omega)\subset \bbR^n$, where
$\Omega$ is a domain in $\bbR^n$, let $D\varphi(x)$ and $J_{\varphi}(x)$ denote the Jacobian matrix
and the Jacobian determinant of $\varphi$ at $x\in \ball$ respectively.

Let $\varphi: \ball\to \ball$ be a non-constant
$C^2$-smooth map with Jacobian not changing sign, and let $\psi$ be a $C^2$ function on $\ball$.
The weighted composition operator $W_{\varphi,\psi}$ is defined by
$W_{\varphi,\psi} f := \psi\cdot (f \circ \varphi)$.
Throughout this paper, for a weighted composition operator $W_{\varphi,\psi}$, we shall always assume that the
Jacobian $J_{\varphi}$ of its inducing map $\varphi$
does not change sign in $\ball$.
This assumption is natural, since for a holomorphic composition operator, the real Jacobian of
its inducing map is always nonnegative.

Our first main result is the following.

\begin{theorem}\label{thm:main1}
Suppose $n\geq 3$ and $1\leq p \leq \infty$. 
Then $W_{\varphi,\psi}$ is bounded on $h^p(\ball)$ if and only if
\begin{enumerate}
\item[(i)]
$\varphi$ is the restriction to $\ball$ of a M\"{o}bius transformation of $\widehat{\mathbb{R}}^n$, and
\item[(ii)]
$\psi$ is a constant multiple of $|D \varphi|^{\frac {n-2}{2}}$.
\end{enumerate}
Here and throughout we write $|D\varphi(x)|:= J_{\varphi}(x)^{\frac {1}{n}}$.
More precisely, $W_{\varphi,\psi}$ is bounded on $h^p(\ball)$ if and only if $\varphi$ has the form
\[
\varphi(x) = b + \frac {\alpha A (x-a)} {|x-a|^{\epsilon}}
\]
and $\psi$ is a constant multiple of $|x-a|^{\frac {\epsilon(2-n)}{2}}$,
where $a\in \mathbb{R}^n\setminus \ball$, $b\in \bbR^n$, $\alpha\in \mathbb{R}$, $A$ is an
orthogonal matrix, and $\epsilon$ is either $0$ or $2$.
\end{theorem}

The case of most interest is when $\varphi$ is a M\"obius transformation of $\ball$ onto $\ball$,
i.e., $\varphi\in\mob(\ball)$. 
In this case,  we obtain a formula for the precise norm of $W_{\varphi,\psi}$, acting
on $h^p(\ball)$. In our next two theorems, we do not assume $n\geq 3$.

\begin{theorem}\label{thm:main3}
Suppose $n\geq 2$ and $1\leq p\leq \infty$. Let $\varphi\in\mob(\ball)$ and $\psi = |D \varphi|^{\frac {n-2}{2}}$. Then
\[
\|W_{\varphi,\psi}\|_{h^p\to h^p}  ~=~ \left(\dfrac {1+|\varphi(0)|}{1-|\varphi(0)|} \right)^{\left|\frac {n-1}{p}-\frac {n-2}{2}\right|}.
\]
\end{theorem}

This may be regarded as a higher dimensional analogue of the well-known result of Nordgren \cite{Nor68}
concerning the norms of composition operators with inner symbols.

%This
%\begin{conjecture}
%Every linear isometry of $h^{\frac {2(n-1)}{n-2}}$ onto $h^{\frac {2(n-1)}{n-2}}$ has the form (1).
%\end{conjecture}

%Note that when $n=2$, this agrees with \cite[ Theorem 3.6 ]{CCMMAC}.
%
In this case, we can also compute the essential norm of $W_{\varphi,\psi}$.
Recall that the essential norm of a bounded linear operator $T$ is the
distance from $T$ to the compact operators, that is,
\[
\|T\|_e := \inf \{ \|T-K\|: K \text{ is compact}\}.
\]
Note that $\|T\|_e = 0$ if and only if $T$ is compact.

\begin{theorem}\label{thm:essentialnorm}
Suppose $n\geq 2$ and $1 <  p < \infty$. Let $\varphi\in\mob(\ball)$ and $\psi = |D \varphi|^{\frac {n-2}{2}}$. Then
\[
\|W_{\varphi,\psi}\|_{e}=\left(\dfrac {1+|\varphi(0)|}{1-|\varphi(0)|} \right)^{\left|\frac {n-1}{p}-\frac {n-2}{2}\right|}.
\]
\end{theorem}

We do not know whether this formula is valid for the cases $p=1$ and $p=\infty$ and leave it as an open question.

The rest of the paper is organized as follows: In Section 2 we recall some basic materials about M\"obius transformations
and the harmonic Hardy spaces. Theorem \ref{thm:main1}, \ref{thm:main3} and \ref{thm:essentialnorm}
will be proved in Sections 3, 4 and 5 respectively.

\section{Preliminaries}
\subsection{Conformal mappings in space and M\"obius transformations}
We recall some basic facts on M\"obius transformations on
$n$-dimensional Euclidean space $\bbR^n$. Good references for these topics
are \cite{Ahl81}, \cite{Bea83} and \cite{Sto16}

Let $\widehat{\mathbb{R}}^n := \bbR^n \cup \{\infty\}$
be the one-point compactification of $\bbR^n$.
The reflection in the unit sphere $\sphere$ is defined by
\[
x ~\longmapsto~ x^{\ast}:= \frac {x}{|x|^2} \quad (0^{\ast} := \infty, \; \infty^{\ast}:=0).
\]
In general, the reflection in the sphere $\sphere(a,r) := \{ x\in \bbR^n: |x-a|=r\}$ is the map
$\varphi: \widehat{\mathbb{R}}^n \to \widehat{\mathbb{R}}^n$
defined by
\begin{equation}\label{eqn:refl}
\varphi(x) ~:=~ a + r^2 (x-a)^{\ast}.
\end{equation}
%\[
%\varphi(x) ~=~ \begin{cases}
%a+ \frac {r^2 (x-a)} {|x-a|^2}, & x\in \bbR^n\setminus \{a\}\\
%\infty, & x=a\\
%a, & x=\infty.
%\end{cases}
%\]
Let $\Pi(a,t):=\{x\in \bbR^n: x\cdot a = t\}$ be a hyperplane in $\bbR^n$,
where $a$ is a unit vector in $\bbR^n$ and $t \in \bbR$.
The reflection in $\Pi(a,t)$ is the map
$\varphi: \widehat{\mathbb{R}}^n \to \widehat{\mathbb{R}}^n$
defined by
\[
\varphi(x) = \begin{cases}
x - 2 ( x\cdot a - t) a^{\ast}, & x\in \bbR^n,\\
\infty, & x=\infty.
\end{cases}
\]
%The map
%$\varphi: \widehat{\mathbb{R}}^n \to \widehat{\mathbb{R}}^n$
%defined by
%\[
%\varphi(x) ~=~ \begin{cases}
%x - 2 (x\cdot a - t) a^{\ast}, & x\in \bbR^n\\
%\infty, & x=\infty
%\end{cases}
%\]
%is called the reflection in $\Pi(a,t)$.
%Also, the reflection in the sphere $\sphere(a,r) := \{ x\in \bbR^n: |x-a|=r\}$ is the map
%$\varphi: \widehat{\mathbb{R}}^n \to \widehat{\mathbb{R}}^n$
%defined by
%\[
%\varphi(x) ~=~ \begin{cases}
%a+ \frac {r^2 (x-a)} {|x-a|^2}, & x\in \bbR^n\setminus \{a\}\\
%\infty, & x=a\\
%a, & x=\infty.
%\end{cases}
%\]

The (general) M\"obius group $\mathrm{M\ddot{o}b}(n)$ is the group of homeomorphisms of $\widehat{\mathbb{R}}^n$
generated by all reflections (in spheres or hyperplanes). Its
members are known as the M\"obius transformations of $\widehat{\mathbb{R}}^n$.

M\"obius transformations are conformal mappings, in the sense that,
for any M\"obius transformation $\varphi$ there is a finite subset $E$ of $\widehat{\mathbb{R}}^n$
such that the restriction of $\varphi$ to the set $\Omega := \widehat{\mathbb{R}}^n\setminus E$ is
a conformal diffeomorphism of $\Omega$ onto $\varphi(\Omega)$.
%In particular, as is easily verified, every reflection $\varphi$ (in a sphere or a hyperplane) satisfies the
In particular, M\"obius transformations satisfy the
Cauchy-Riemann system
\begin{equation}\label{eqn:CRsys}
D\varphi(x) (D\varphi(x))^t = |D\varphi(x)|^2 I_n
\end{equation}
for all $x\in \Omega$. It should be noted that the notion of the
Cauchy-Riemann system defined here is slightly different from that in
\cite{IM01}, because all vectors in this paper will be column vectors.
%so too are all M\"obius transformations of $\widehat{\mathbb{R}}^n$.

Conversely, according to the generalized Liouville theorem proved in 1993 by T. Iwaniec and G.J. Martin,
the only nonconstant solutions to the Cauchy-Riemann system of Sobolev class $W^{1,n}_{loc}(\Omega, \bbR^n)$
are the restrictions to $\Omega$ of M\"obius transformations of $\widehat{\mathbb{R}}^n$.
Let $\Omega$ be a domain in $\bbR^n$.  By a weak solution of the Cauchy-Riemann System on $\Omega$ we mean any
function $\varphi: \Omega \to \bbR^n$ of Sobolev class $W^{1,n}_{loc}(\Omega, \bbR^n)$ such that
\begin{itemize}
\item
$\varphi$ is sense-preserving or sense-reversing;
\item
\eqref{eqn:CRsys} holds for almost every $x\in \Omega$.
\end{itemize}

\begin{lemma}[{\cite[Theorem 5.1.1]{IM01}}]\label{lem:Liouville}
Suppose $n\geq 3$.  Every weak solution $\varphi: \Omega \to \bbR^n$ of
the Cauchy-Riemann system \eqref{eqn:CRsys} is either constant or the restriction to $\Omega$ of a
M\"obius transformations of $\widehat{\mathbb{R}}^n$. More precisely, $\varphi$ has the form
\[
\varphi(x) = b + \frac {\alpha A (x-a)} {|x-a|^{\epsilon}},
\]
where $a\in \mathbb{R}^n\setminus \Omega$, $b\in \bbR^n$, $\alpha\in \mathbb{R}$, $A$ is an
orthogonal matrix, and $\epsilon$ is either $0$ or $2$.
\end{lemma}

%A M\"obius transformation of $\widehat{\mathbb{R}}^n$ (the one-point compactification of $\mathbb{R}^n$) is a finite composition
%of reflections in spheres or hyperplanes. The group of M\"obius transformations on $\widehat{\mathbb{R}}^n$ is
%called the general M\"obius group and is denoted by $\mathrm{GM}(\widehat{\mathbb{R}}^n)$.
We denoted by $\mob(\ball)$ the subgroup of $\mathrm{M\ddot{o}b}(n)$ which keeps $\ball$ invariant;
in other words, $\mob(\ball)$ consists of all M\"obius transformations of $\ball$ onto $\ball$.
To each point $a\in \ball$ we may associate a unique $\varphi_a \in \mob(\ball)$, enjoying
the following properties:
\begin{enumerate}
\item[(i)]
$\varphi_a(0)=a, \; \varphi_a(a)=0$;
\item[(ii)]
$\varphi_a^{-1} = \varphi_a$;
\item[(iii)]
$\varphi_a$ has a unique fixed point.
\end{enumerate}
Such a M\"obius transformation can be expressed by an explicit formula
\begin{equation}\label{eqn:moebius}
\varphi_a(x)~:=~\frac{|x-a|^2a-(1-|a|^2)(x-a)}{[x,a]^2},
\end{equation}
with
\[
[x,a] := (1-2x\cdot a+|x|^2|a|^2)^{1/2}.
\]
The M\"obius transformation $\varphi_b\circ \varphi_a$ carries the point $a\in \ball$ to the point $b\in \ball$. Thus the
group $\mob(\ball)$ is transitive on $\ball$.
Furthermore,  any $\varphi \in \mob(\ball)$, $\varphi\neq I$,
has a unique representation of the form:
\begin{equation}
\varphi ~=~ A \circ \varphi_a,
\end{equation}
where $A$ is an orthogonal transformation and $a=\varphi^{-1}(0)$.
%For any point $a\in \ball$, we define
%\begin{equation}\label{eqn:moebius}
%\varphi_a(x)~:=~\frac{|x-a|^2a-(1-|a|^2)(x-a)}{[x,a]^2},
%\end{equation}
%with
%\[
%[x,a] := (1-2x\cdot a+|x|^2|a|^2)^{1/2}.
%\]
%%By a straightforward computation we have
%%\[
%%1-|\varphi_a(x)|^2 ~=~ \frac {(1-|x|^2)(1-|a|^2)}{[x,a]^2}.
%%\]
%Then $\varphi_a \in \mob(\ball)$ and satisfies
%\begin{equation}
%\varphi_a(0)=a, \quad \varphi_a(a)=0, \quad \text{and} \quad  \varphi_a(\varphi_a(x))=x
%\end{equation}
%for all $x\in \ball$. Furthermore,  every $\varphi \in \mathrm{M}(\ball)$, $\varphi\neq I$, has a canonical representation
%\begin{equation}
%\varphi ~=~ A \, \varphi_a,
%\end{equation}
%with $a=\varphi^{-1}(0)$ and $A$ an orthogonal transformation.

%We refer to [10, Chapter II] for the properties of these transformations.
We present three convenient formulas, in which $\varphi \in \mob(\ball)$,
\begin{align}
|D\varphi(x)| ~=~& \frac{1-|\varphi^{-1}(0)|^2}{[x, \varphi^{-1}(0)]^2}, \label{eqn:moebiden1} \\
1-|\varphi(x)|^2 ~=~& |D\varphi(x)| (1-|x|^2), \label{eqn:moebiden2}\\
%\intertext{and}
|\varphi(x)-\varphi(y)| ~=~& |D\varphi(x)|^{\frac {1}{2}} \, |D\varphi(y)|^{\frac {1}{2}} |x-y|,
\quad (x, y\in \overline{B}). \label{eqn:moebiden3}
\end{align}
Also, \eqref{eqn:moebiden1}, together with the easy observation that $|\varphi(0)| ~=~ |\varphi^{-1}(0)|$,
implies
\begin{equation}\label{eqn:jac-est}
\frac {1-|\varphi(0)|}{1+|\varphi(0)|}~\leq~ |D\varphi(x)|~\leq~ \frac{1+|\varphi(0)|}{1-|\varphi(0)|}
\end{equation}
and
\begin{equation}\label{eqn:jac-est2}
\frac {|D\varphi(y)|} {|D\varphi(x)|}\leq \left(\frac{1+|\varphi(0)|}{1-|\varphi(0)|}\right)^{2}.
\end{equation}

%It is obvious from (x.x) and (x.x) that
%\[
%|\varphi(0)| ~=~ |\varphi^{-1}(0)|.
%\]
%
%%\item[(2)] $\displaystyle |D\varphi(x)|=\frac{1-|\varphi^{-1}(0)|^2}{[x, \varphi^{-1}(0)]^2}.$\\
%%
%%\item[(3)]$\displaystyle |\varphi(x)-\varphi(y)| = |D\varphi(x)|^{\frac {1}{2}} \, |D\varphi(y)|^{\frac {1}{2}} |x-y|$. \\
%%
%%\item[(4)]$\displaystyle  1-|\varphi(x)|^2 = |D\varphi(x)| (1-|x|^2). $\\
%%
%%
%%\item[(5)]$\displaystyle \frac{|\varphi(x)-\varphi(y)|}{1-|\varphi(x)|^2} ~=~
%%\left(\frac {|D\varphi(y)|} {|D\varphi(x)|} \right)^{\frac {1}{2}} \frac {|x-y|} {1-|x|^2}$.
%%
%%\item[(6)]For any $x, y\in \overline{\ball}$, we have
%%\[
%%\frac {1-|\varphi(0)|}{1+|\varphi(0)|}~\leq~ |D\varphi(x)|~\leq~ \frac{1+|\varphi(0)|}{1-|\varphi(0)|},
%%\]
%\[
%\frac {1-|\varphi(0)|}{1+|\varphi(0)|} ~\leq~ |D\varphi^{-1}(x)| ~\leq~ \frac {1+|\varphi(0)|}{1-|\varphi(0)|},
%\]
%and
%\[
%\left(\frac {|D\varphi(y)|} {|D\varphi(x)|} \right)^{\frac {1}{2}}\leq \frac{1+|\varphi(0)|}{1-|\varphi(0)|}.
%\]
%
%$\displaystyle |D\varphi(\varphi^{-1}(\zeta))|=|D\varphi^{-1}(\zeta)|^{-1}, \quad \quad |D\varphi^{-1}(\varphi(\zeta))|=|D\varphi(\zeta)|^{-1}$.
%
%%\end{itemize}
%%
%%
%%

\subsection{Harmonic Hardy spaces}

For the material in this section we refer the reader to \cite[Chapter 6]{ABR01}.

As it has been defined in the introduction, for $1\leq p<\infty$, the harmonic Hardy space $h^p(\ball)$
consists of all complex-valued harmonic function $f$ on $\ball$ such that
\[
\|f\|_{h^p}:= \sup_{0\leq r<1} \bigg\{\int\limits_{\sphere} |f(r\zeta)|^p d\sigma(\zeta)\bigg\}^{\frac {1}{p}} < \infty.
\]
Also, $h^\infty$ is the Banach space of complex-valued bounded harmonic function on $\ball$.

%We recall some basic facts about the harmonic Hardy space $h^p$, from \cite[Chapter 6]{ABR}.
For a complex Borel measure $\mu$ on $\sphere$, its Poisson integral is the function
\[
P[\mu](x):=\int\limits_{\sphere} P(x,\zeta) d\mu(\zeta), \quad x\in \ball,
\]
where
\[
P(x,\zeta):=\frac {1-|x|^2}{|x-\zeta|^n}, \quad (x,\zeta)\in \ball\times \sphere,
\]
is the Poisson kernel. It is well known that a harmonic function belongs to $h^1(\ball)$
if and only if it is the Poisson integral of a complex Borel measure. Moreover, for
$1<p\leq \infty$, the space $h^p(\ball)$ coincides with the space of Poisson integrals of
$L^p(\sphere)$ functions.

Given $\delta>1$ and $\zeta\in \sphere$, let $\Gamma_{\delta}(\zeta)$ be the nontangential approach region
with vertex $\zeta$ consisting of all points $x\in \ball$ such that
\[
|x-\zeta|< \frac {\delta}{2} (1-|x|^2).
\]
A function $f$ on $\ball$ is said to have nontangential limit $L$ at $\zeta\in \sphere$ if
for each $\delta>1$, we have $f(x)\to L$ as $x\to \zeta$ within $\Gamma_{\delta}(\zeta)$.

If $1<p<\infty$ and $f\in h^p(\ball)$, then, by the Fatou theorem (see \cite[p. 137]{ABR01}), $f$ has nontangential
limits $f^{\ast}$ almost everywhere on $\sphere$. (The use of the superscript $\ast$ here should
cause no confusion with the previous use of that as the reflection in the unit sphere.)
We shall call $f^{\ast}$ the boundary function of $f$.
The harmonic Hardy space $h^p(\ball)$ can be identified with a closed subspace
of $L^p(\sphere)$ via the isometry $f\to f^{\ast}$. In particular,
$\| f\|_{h^p} = \|f^{\ast}\|_{L^p(\sphere)}$ for every $f\in h^p(\ball)$.

%The Fatou theorem states that if $f\in h^p(\ball)$ with $1<p<\infty$ then $f$ has nontangential limits
%almost everywhere on $\sphere$. Moreover, $h^p(\ball)$ can naturally be identified with a closed subspace
%of $L^p(\sphere)$. More precisely,
%\begin{lemma}
%Suppose $f\in h^p(\ball)$ with $1<p<\infty$. Then the nontangential limit
%\[
%f^{\ast}(\zeta):= \lim_{\substack{x\to \zeta\\
%x\in \Gamma_{\delta}(\zeta)}} f(x)
%\]
%exists for almost all $\zeta\in \sphere$ and $\|f\|_{h^p} = \| f^{\ast}\|_{L^p(\sphere)}$.
%%\[
%%\|f\|_{h^p} ~=~ \bigg\{\int\limits_{\sphere} |f^\ast(\zeta)|^p d\sigma(\zeta)\bigg\}^{\frac {1}{p}}.
%%\]
%\end{lemma}

%if $1<p\leq\infty$ and $f\in h^p$ then $f = P[f^{\ast}]$, the Poisson integral of $f^{\ast}$.
%Since the map $f^{\ast} \to P[f^{\ast}]$ is a linear isometry of $L^p(\sphere)$ into $h^p$, we have
%\[
%\|f\|_{h^p}=\|f^\ast\|_p,
%\]
%where $\|\cdot\|_p := \|\cdot\|_{L^p(\sphere)}$.

%\begin{lemma}\label{lem:norm}
%Suppose $1<p\leq\infty$ and $f\in h^p$. Then
%\begin{equation}\label{eq:hpnorm}
%\|f\|_{h^p}=\|f^\ast\|_p
%\end{equation}
%where $f^\ast(\zeta):=\lim\limits_{r\to 1^-}f(r\zeta)$ denotes the radial limit of $f$ at $\zeta\in \sphere$.
%\end{lemma}
%\begin{proof}
%This is immediate from \cite[theorem 6.13]{SPW} and \cite[theorem 6.39 ]{SPW}.
%\end{proof}
%

\begin{lemma}\label{lem: aast}
If $1<p\leq \infty$, $f\in h^p(\ball)$ and $\varphi\in \mob(\ball)$, then
\[
(f\circ\varphi)^\ast(\zeta)=f^\ast(\varphi(\zeta))
\]
for almost every $\zeta\in \sphere$.
\end{lemma}

\begin{proof}
It is equivalent to show that $(f\circ \varphi)^{\ast} (\varphi^{-1}(\zeta)) = f^{\ast}(\zeta)$ for almost every $\zeta\in \sphere$ (by replacing $\zeta$ by $\varphi^{-1}(\zeta)$).
Assume that $f$ has a nontangential limit at $\zeta\in \sphere$.
It suffices to show that for any $\delta>1$,
\[
\varphi(\Gamma_{\delta}(\zeta)) ~\subset~ \Gamma_{\tilde{\delta}}(\varphi(\zeta)), \quad \text{with }
\tilde{\delta}:= \frac {1+|\varphi(0)|}{1-|\varphi(0)|} \, \delta.
\]
Assume $x\in \Gamma_{\delta}(\zeta)$. Using \eqref{eqn:moebiden2}, \eqref{eqn:moebiden3} and \eqref{eqn:jac-est2}, we get
\begin{align*}
\frac{|\varphi(x)-\varphi(\zeta)|}{1-|\varphi(x)|^2} ~=~&
\frac {|x-\zeta|} {1-|x|^2} \left(\frac {|D\varphi(\zeta)|} {|D\varphi(x)|} \right)^{\frac {1}{2}}
%~=~& \frac {[x, \varphi^{-1}(0)]} {[\zeta, \varphi^{-1}(0)]} \frac {|x-\zeta|} {1-|x|^2}. \notag\\
~\leq \frac {\delta}{2} \left\{\frac {1+|\varphi(0)|}{1-|\varphi(0)|}\right\},
\end{align*}
%Since $x\in \Gamma_{\delta}(\zeta)$,
%\[
%\frac {|x-\zeta|} {1-|x|^2} \leq \frac {\delta}{2}.
%\]
%Also, it is easy to see that
%\[
%\frac {[x, \varphi^{-1}(0)]} {[\zeta, \varphi^{-1}(0)]} ~\leq~ \frac {1+|\varphi^{-1}(0)|}{1-|\varphi^{-1}(0)|} ~=~
%\frac {1+|\varphi(0)|}{1-|\varphi(0)|}.
%\]
%%It follows that
%\[
%|\varphi(x)-\varphi(\zeta)| ~\leq~ \frac {\delta}{2} \left\{\frac {1+|\varphi(0)|}{1-|\varphi(0)|}\, \delta\right\} (1-|\varphi(x)|^2),
%\]
%by the facts
%\[
%\bigg(\frac {|D\varphi(\zeta)|} {|D\varphi(x)|} \bigg)^{1/2}\leq \frac {1+|\varphi(0)|}{1-|\varphi(0)|} \quad
%\text{and} \quad \frac {|x-\zeta|} {1-|x|^2} \leq \frac {\delta}{2}.
%\]
which implies that  $\varphi(x) \in \Gamma_{\tilde{\delta}} (\varphi(\zeta))$. The proof is complete.
\end{proof}

The following change-of-variable formula is well-known.  We give a proof of it for completeness.

\begin{lemma}\label{lem:Jacobiformula}
If $f\in L^{1}(\sphere)$ and $\varphi\in \mob(\ball)$, then
\begin{equation}\label{eq:Jacobian}
\int\limits_{\sphere} f(\varphi(\zeta)) d\sigma(\zeta) = \int\limits_{\sphere} f(\zeta)|D\varphi^{-1}(\zeta)|^{n-1}d\sigma(\zeta).
\end{equation}
\end{lemma}

%Although this formula is well known, we give a brief derivation of it for
%completeness.

\begin{proof}
By an approximation argument, we may assume that $f$ is continuous on $\sphere$.
We denote by $P[f]$ the Poisson integral of $f$. Since $W_{\varphi,\psi} (P[f])$ is harmonic
in $\ball$, for every $0\leq r<1$,
\begin{align}\label{eqn:mean-value}
\int_{\sphere} P[f](\varphi(r\zeta)) |D\varphi(r\zeta)|^{\frac {n-2}{2}} d\sigma(\zeta)
&~=~\int\limits_{\sphere} W_{\varphi,\psi} (P[f])(r\zeta)) d\sigma(\zeta) \\
&~=~(1-|\varphi(0)|^2)^{\frac {n-2}{2}} P[f](\varphi(0)), \notag
\end{align}
where the last equality follows from the mean-value property of harmonic functions and the formula
\[
|D\varphi (0)| ~=~ 1-|\varphi(0)|^2.
\]
Note that $\varphi(r\zeta) \to \varphi(\zeta)$ as $r\to 1$, and
\[
\lim\limits_{r\to 1}  P[f](\varphi(r\zeta)) = P[f](\varphi(\zeta)) = f(\varphi(\zeta))
\]
for each $\zeta\in \sphere$.
So, passing to the limit in \eqref{eqn:mean-value} yields
\begin{align*}
\int\limits_{\sphere} f(\varphi(\zeta)) |D\varphi(\zeta)|^{\frac {n-2}{2}}d\sigma(\zeta)
~=~& (1-|\varphi(0)|^2)^{\frac {n-2}{2}}P[f](\varphi(0))\\
=~& \int\limits_{\sphere}f(\zeta) |D\varphi^{-1}(\zeta)|^{\frac {n}{2}}d\sigma(\zeta).
\end{align*}
Now, replacing $f$ by $f\cdot |D\varphi^{-1}|^{\frac {n-2}{2}}$ in the above equality yields \eqref{eq:Jacobian}.
\end{proof}

\subsection{Extended Poisson Kernel}
For fixed $y\in \ball$, put
\[
P_y(x) := \frac{1-|x|^2|y|^2}{(1-2x\cdot y+|x|^2|y|^2)^{n/2}}, \quad x\in \ball.
\]
Note that $(P_y)^{\ast} =P(y, \cdot)$ on $\sphere$, where $P(y, \zeta)$ is the Poisson kernel for the unit ball.
It is easy to check that $P_y\in h^p(\ball)$ for $1\leq p \leq \infty$. Furthermore, we have the following

\begin{lemma}\label{lem:thenormofposs}
If $1\leq  p\leq \infty$ then
\begin{equation}\label{eq:normpoisson}
\|P_y\|_{h^p}= \Phi_p(|y|^2)\, (1-|y|^2)^{\frac {1-n}{p^\prime}},
\end{equation}
where $p^{\prime}$ is the conjugate exponent of $p$, and
\[
\Phi_p(r): = \begin{cases} \left\{\hyperg{ \frac {n(1-p)}{2}} {\frac {2n-2-np}{2}}{\frac {n}{2}}{r}\right\}^{\frac {1}{p}}, &
1\leq p<\infty,\\[16pt]
(1+r)^n, & p=\infty.
\end{cases}
\]
Moreover, we have
\begin{equation}\label{eq:limitpoisson}
\lim_{r\to 1^{-}} \Phi_p(r) =  \begin{cases} \bigg\{
\dfrac{\Gamma(\frac {n}{2})\Gamma(np+1-n)}{\Gamma(\frac {np}{2})\Gamma(\frac {np+2-n}{2})}\bigg\}^{\frac {1}{p}},& 1\leq p<\infty,\\[12pt]
2^n,& p=\infty.
\end{cases}
\end{equation}
%When $p=\infty$, the quantity on the right-hand side of \eqref{eq:limitpoisson} should be interpreted as $2^n$.
\end{lemma}

%\begin{lemma}\label{lem:thenormofposs}
%Suppose $1< p\leq\infty$. Then
%\begin{enumerate}
%\item If $1<p<\infty$, then  let $p^\prime=p/(p-1)$, we have
%\begin{equation}\label{eq:normpoisson}
%\|P_y\|_{h^p}=\frac{\big(\Phi_p(|y|^2)\big)^{\frac {1}{p}}}{(1-|y|^2)^{(n-1)/p^\prime}},
%\end{equation}
%with
%\[
%\Phi_p(r): =\hyperg{n(1-p)/2}{n-1-\frac {np}{2}}{\frac {n}{2}}{r}.
%\]
%Moreover, we have
%\begin{equation}\label{eq:limitpoisson}
%\lim_{|y|\to 1^{-}}(1-|y|^2)^{\frac{n-1}{p^\prime}}\|P_y\|_{h^p}=\bigg(
%\frac{\Gamma(\frac {n}{2})\Gamma(np+1-n)}{\Gamma(\frac {np}{2})\Gamma(\frac {np}{2}+1-\frac {n}{2})}\bigg)^{\frac {1}{p}}.
%\end{equation}
%\item If $p=\infty$, then we have
%\[
%\|P_y\|_{h^\infty}=\frac{1+|y|}{(1-|y|)^{n-1}}.
%\]
%Moreover, we have
%\[
%\lim_{|y|\to 1}(1-|y|^2)^{n-1}\|P_y\|_{h^\infty}=2^n.
%\]
%\end{enumerate}
%\end{lemma}

\begin{proof}
Since the case $p=\infty$ is trivial, we shall assume that $1\leq p<\infty$. We recall the following formula from \cite[Lemma 2.1]{LP04}:
\begin{equation}\label{eqn:integration-on-sphere1}
\int\limits_{\sphere} \frac {d\sigma(\zeta)} {|x-\zeta|^{2s}} =
\hyperg {s} {s-\frac {n}{2} +1} {\frac {n}{2}} {|x|^2}.
\end{equation}
It follows that
\begin{align*}
\|P_y\|_{h^p}^p~=~&\int\limits_{\sphere}\frac{(1-|y|^2)^p}{|y-\zeta|^{np}}d\sigma(\zeta)\\
~=~&(1-|y|^2)^p\hyperg{\frac {np}{2}} {\frac {np-n+2}{2}} {\frac {n}{2}}{|y|^2}\\
=~&(1-|y|^2)^{p+n-1-np} \hyperg{\frac {n(1-p)}{2}} {\frac {2n-2-np}{2}}{\frac {n}{2}}{|y|^2},
\end{align*}
which is exactly \eqref{eq:normpoisson}.
In the last equality we used the formula
\begin{equation*}\label{eqn:euler}
\hyperg{a}{b}{c}{z} = (1-z)^{c-a-b} \hyperg{c-a}{c-b}{c}{z}.
\end{equation*}
Also, \eqref{eq:limitpoisson} is immediate from \eqref{eq:normpoisson} and the following well-known formula
\begin{equation*}\label{eqn:gauss}
\hyperg{a}{b}{c}{1}=\frac {\Gamma(c) \Gamma(c-a-b)} {\Gamma(c-a)
\Gamma(c-b)}.
\end{equation*}
%If $p=\infty$, then the result comes from the following two sides estimates. On one hand,  for any fixed $y\in \ball$, we have
%\[
%\|P_y\|_{h^\infty}=\sup_{x\in \ball}\frac{1-|x|^2|y|^2}{[x, y]^n}\leq\sup_{x\in \ball}\frac{1+|x||y|}{(1-|x||y|)^{n-1}}\leq \frac{1+|y|}{(1-|y|)^{n-1}}.
%\]
%On the other hand, we have
%\[
%\|P_y\|_{h^\infty}=\sup_{x\in \ball}\frac{1-|x|^2|y|^2}{[x, y]^n}\geq \lim_{r\to 1^{-}}\frac{1-r^2\big|y/|y|\big|^2|y|^2}{[ry/|y|, y]^n}= \frac{1+|y|}{(1-|y|)^{n-1}}.
%\]
%
%
The proof is complete.
\end{proof}

\section{Proof of Theorem \ref{thm:main1}}

\subsection{Necessity}

To prove the necessity, it suffices to show that if $W_{\varphi,\psi}$ sends $h^p(\ball)$ into itself then
the conditions (i) and (ii) in Theorem \ref{thm:main1} are satisfied.
We shall in fact prove the following slightly stronger result.

\begin{proposition}\label{prop:harmcty}
Suppose $n\geq 3$. Let $\varphi: \ball\to \ball$ be a $C^2$ map with Jacobian
not changing sign, and let $\psi$ be a $C^2$ function on $\ball$. Then
$W_{\varphi,\psi}$ preserves harmonic functions if and only if the conditions
(i) and (ii) in Theorem \ref{thm:main1} are satisfied.
\end{proposition}

%We divide the proof into a sequence of lemmas.

%When operating with matrices we treat $x\in \mathbb{R}^n$ as a column vector.

%We begin with the following lemma.
\begin{lemma}\label{thm:app1}
Let $\Omega$ be a domain in $\bbR^n$.  Let $\varphi: \Omega \to \varphi(\Omega) \subset \bbR^n$
be a $C^2$ map and $\psi$ be a $C^2$ function on $\Omega$.
Then $W_{\varphi,\psi}$ preserves harmonic functions if and only if
the pair $\{\varphi, \psi\}$ satisfies
\begin{subequations}\label{E:gp}
\begin{gather}
\Delta \psi = 0, \label{E:gp1}\\
\psi\Delta\varphi + 2(D\varphi) (\nabla \psi) = 0, \label{E:gp2}\\
D\varphi(D\varphi)^t = |D\varphi|^2 I_n \label{E:gp3}.
\end{gather}
\end{subequations}
\end{lemma}

\begin{proof}
A lengthy, but straightforward calculation yields
% \begin{align}\label{eq:deltag}
%&\Delta[W_{\varphi, \psi}f]\\
%=~&W_{\varphi,\Delta \psi}f+\big\langle\nabla
%f, \quad \psi\Delta\varphi +2D\varphi\nabla
%\psi\big\rangle+\psi\cdot\trace \big((D\varphi)^\tau\cdot Hf\cdot D\varphi\big)\notag.
%\end{align}
\begin{equation}\label{eq:deltag}
\Delta (W_{\varphi, \psi} f) = \psi\, \trace \big\{(D\varphi)^t (Hf) (D\varphi) \big\}
+ \nabla f \cdot \big\{\psi \Delta\varphi +2 (D\varphi) \nabla \psi \big\} + W_{\varphi,\Delta \psi}f,
%=~&W_{\varphi,\Delta \psi}f+\notag.
\end{equation}
where
\[
Hf:=\left(\frac{\partial^2 f}{\partial x_i \partial x_j} \right)_{1\leq i, j\leq n}
\]
is the Hessian matrix of the function $f$ and $\Delta\varphi=\left(\Delta\varphi^{(1)}, \ldots, \Delta\varphi^{(n)}\right)^{t}$.
It is then easy to see that the conditions \eqref{E:gp1}--\eqref{E:gp3} imply that
\[
\Delta (W_{\varphi, \psi}f) = \psi\,  |D\varphi|^2\, \Delta f.
\]
This proves the ``if'' part of our statement.

To prove the ``only if'' part, we first take $f\equiv 1$ in \eqref{eq:deltag}. Then the assumption $\Delta (W_{\varphi, \psi}f)=0$
immediately yields \eqref{E:gp1}.

Next, we take $f(x)=x_i$ in \eqref{eq:deltag}, for $i=1,2,\cdots,n$. The assumption $\Delta (W_{\varphi, \psi}f)=0$,
along with \eqref{E:gp1} then yields \eqref{E:gp2}.

Now we take $f(x)=x_i^2-x_j^2$ in \eqref{eq:deltag}, for $i,j=1,2,\cdots,n$. Then, together with \eqref{E:gp1} and \eqref{E:gp2},
the assumption $\Delta (W_{\varphi, \psi}f)=0$ implies
\[
\sum_{k=1}^{n} \left\{ \left(\frac {\partial\varphi^{(i)}} {\partial x_k}\right)^2
- \left(\frac {\partial\varphi^{(j)}} {\partial x_k}\right)^2 \right\} = 0.
\]
Namely,
\begin{equation}\label{eq:fact11}
|\nabla \varphi^{(i)}|=|\nabla \varphi^{(j)}| \quad \text{for all } i,j\in \{1,2,\cdots,n\}.
\end{equation}

Finally, we take $f(x)=x_i x_j$ in \eqref{eq:deltag}, for $i,j\in \{1,2,\cdots,n\}$ with $i\neq j$.
Then it follows from \eqref{E:gp1} and \eqref{E:gp2} and the assumption $\Delta (W_{\varphi, \psi}f)=0$ that
\[
\sum_{k=1}^{n} \left(\frac {\partial\varphi^{(i)}} {\partial x_k}\right)
\left(\frac {\partial\varphi^{(j)}} {\partial x_k}\right) = 0,
\]
that is,
\begin{equation}\label{eq:fact22}
\nabla \varphi^{(i)} \cdot \nabla \varphi^{(j)} = 0 \quad \text{for all } i,j\in \{1,2,\cdots,n\} \text{ with } i\neq j.
\end{equation}
This, together with \eqref{eq:fact11}, gives \eqref{E:gp3}, and the proof is complete.
\end{proof}

\begin{proof}[\textbf{Proof of Proposition \ref{prop:harmcty}}]

To prove the ``if'' part, it suffices to verify \eqref{E:gp1}-\eqref{E:gp3}
for all reflections (in spheres or hyperplanes) $\varphi$ and $\psi=|D\varphi|^{\frac {n-2}{2}}$.
The verifications \eqref{E:gp1}-\eqref{E:gp3} for reflections in hyperplanes
are trivial.  We sketch the calculations for the reflection \eqref{eqn:refl} in the sphere
$\sphere(a,r)$. More explicitly,
\[
\varphi(x) = a + \frac {r^2} {|x-a|^2} (x-a)
\]
with $a \in \bbR^n\setminus \overline{B}$ (otherwise contradicts $\varphi(\ball) \subset \ball$,
since $\varphi(a)=\infty$). Since $\varphi$ is conformal, it trivially satisfies \eqref{E:gp3}.
Also, it is easy to show that
\[
\psi (x) =|D\varphi(x) |^{\frac {n-2}{2}} = \frac {r^{n-2}} {|x-a|^{n-2}}
\]
and hence $\Delta \psi (x) = 0$ for all $x\in \ball$, so \eqref{E:gp1} holds.
Finally, straightforward calculations yield
\begin{align*}
\Delta \varphi(x) ~=~& \frac {2(2-n) r^2} {|x-a|^4} (x-a), \\
\nabla \psi (x) ~=~& \frac {(2-n) r^{n-2}}{|x-a|^n} (x-a),\\
D\varphi (x) ~=~& \frac {r^2} {|x-a|^{2}} [I_n - 2 Q(x-a)],
\end{align*}
where $Q(x)$ is the matrix with entries $x_i x_j/|x|^2$, and \eqref{E:gp2} easily follows.

Now we suppose that $W_{\varphi, \psi}$ preserves harmonic functions.  Then, by Lemma \ref{thm:app1},
the pair $\{\varphi, \psi\}$ satisfies \eqref{E:gp1}-\eqref{E:gp3}. In particular, \eqref{E:gp3},
by Lemma \ref{lem:Liouville}, implies that
$\varphi$ is the restriction to $\ball$ of a M\"{o}bius transformation of $\widehat{\mathbb{R}}^n$.

%is either the restriction to $\ball$ of an affine transformation or a M\"{o}bius transformation of $\ball$.

%In the case $\epsilon=0$, $\varphi$ is an affine mapping of the form \eqref{eqn:affine}, and it is immediate from \eqref{E:gp2} that $\psi$ is constant.
%
%Now we let $\epsilon=2$. $\varphi$ be a M\"{o}bius transformation of $\ball$.
Next, it follows from \eqref{E:gp2} that
\[
\psi (D\varphi)^t \Delta\varphi + 2(D\varphi)^t (D\varphi) (\nabla \psi) = 0.
\]
Combing with \eqref{E:gp3}, this gives
\[
\nabla \psi = - \frac {\psi}{2|D\varphi|^2} (D\varphi)^t \Delta \varphi.
\]
On the other hand, a straightforward calculation shows
\[
- \frac {1}{2|D\varphi|^2} (D\varphi)^t \Delta \varphi ~=~ \frac {n-2}{2}\, \nabla \left(\log |D\varphi|\right).
\]
It follows that
\[
\nabla \psi ~=~ \frac {n-2}{2}\, \psi\ \nabla \left(\log |D\varphi|\right),
\]
which in turn shows that
\[
\psi=C|D\varphi|^{\frac {n-2}{2}}
\]
for some constant $C$. The proof is complete.
\end{proof}

\subsection{Sufficiency}

%\begin{proof}

Assume that the conditions (i) and (ii) in Theorem \ref{thm:main1} hold.
In view of the ``if'' part of Proposition \ref{prop:harmcty}, we only need to show that
$W_{\varphi,\psi}$ is bounded on $h^p(\ball)$.

We first show that $|D\varphi|\in L^{\infty}(\ball)$.
Recall that $\varphi$ has the form
\[
\varphi(x) = b + \frac {\alpha A (x-a)} {|x-a|^{\epsilon}}.
\]
where $a\in \mathbb{R}^n\setminus \ball$, $b\in \ball$, $\alpha\in \mathbb{R}$, $A$ is an
orthogonal matrix, and $\epsilon$ is either $0$ or $2$.
The case $\epsilon=0$
is trivial. When $\epsilon=2$, it is clear that $\varphi(a)=\infty$. Since $\varphi(\ball)\subset \ball$, we must have $|a|>1$.
Thus,
\[
|D\varphi(x)| ~=~ \frac {|\alpha|} {|x-a|^2} ~\leq~ \frac {|\alpha|} {(|a|-1)^2}
\]
for all $x\in \ball$.

Next, we show that the inequality
\begin{equation}\label{eqn:Poisphi}
\int\limits_{\sphere} P(\varphi(r\zeta),\eta) \, |D\varphi(r\zeta)|^{\frac {n-2}{2}} d\sigma(\zeta)
~\leq~ \frac {1+|\varphi(0)|} {(1-|\varphi(0)|)^{n-1}}  |D\varphi(0)|^{\frac {n-2}{2}}
\end{equation}
holds for every $r\in [0,1)$ and every $\eta\in \sphere$. %Here,
%\[
%P(x,\eta) := \frac {1-|x|^2}{|x-\eta|^n}
%\]
%is the Poisson kernel for the unit ball $\ball$.

Indeed, by Theorem \ref{thm:main1}, for each fixed $\eta\in \sphere$, the function
\[
x ~\longmapsto~ P(\varphi(x),\eta) |D\varphi(x)|^{\frac {n-2}{2}}
\]
is harmonic in $\ball$. Then, for every $r\in (0,1)$,
\[
\int\limits_{\sphere} P(\varphi(r\zeta),\eta) \, |D\varphi(r\zeta)|^{\frac {n-2}{2}} d\sigma(\zeta)
~=~ P(\varphi(0),\eta) |D\varphi(0)|^{\frac {n-2}{2}},
\]
by the mean-value property of harmonic functions. The inequality \eqref{eqn:Poisphi} then follows, since
\[
\sup_{\eta \in \sphere} P(\varphi(0),\eta) ~\leq~ \frac {1+|\varphi(0)|} {(1-|\varphi(0)|)^{n-1}}.
\]

The rest of the proof is divided into three cases.

\subsubsection*{Case I: $p=1$}

Let $f\in h^1(\ball)$. By \cite[theorem 6.13]{ABR01}, there exists a complex measure $\mu$ on $\sphere$ such that $f=P[\mu]$
and $\|f\|_{h^1}=\|\mu\|$. It follows that
%\begin{align*}
%|f(\varphi(r\zeta))|~=~&|P[\mu](\varphi(r\zeta))|\\
%\leq~&\int_{\sphere}P(\varphi(r\zeta), \eta)d|\mu|(\eta).
%\end{align*}
%Then by Fubini's theorem
\begin{align*}
\int_{\sphere}&|f(\varphi(r\zeta))||D\varphi(r\zeta)|^{\frac {n-2}{2}}d\sigma(\zeta)\notag\\
&~\leq~ \int_{\sphere} \bigg\{ \int_{\sphere} P(\varphi(r\zeta), \eta)d|\mu|(\eta) \bigg\} |D\varphi(r\zeta)|^{\frac {n-2}{2}}d\sigma(\zeta)\notag\\
&~=~ \int_{\sphere}\bigg\{\int_{\sphere}P(\varphi(r\zeta), \eta)|D\varphi(r\zeta)|^{\frac {n-2}{2}}d\sigma(\zeta)\bigg\}d|\mu|(\eta)\\
&~\leq~ \bigg\{ \frac {1+|\varphi(0)|} {(1-|\varphi(0)|)^{n-1}}  |D\varphi(0)|^{\frac {n-2}{2}}\bigg\}\, \|\mu\|.
\end{align*}
where in the last line we used \eqref{eqn:Poisphi}. This gives
\begin{equation}\label{eqn:upperest1}
\| W_{\varphi,\psi} \|_{h^1\to h^1} ~\leq~ \frac {1+|\varphi(0)|} {(1-|\varphi(0)|)^{n-1}}  |\psi(0)|.
\end{equation}

\subsubsection*{Case II: $1< p<\infty$}
Let $f\in h^p(\ball)$. Then by \cite[Theorems 6.13 and 6.39]{ABR01}, $f=P[f^{\ast}]$ for some $f^{\ast} \in L^p(\sphere)$.
Thus,
\[
|f(\varphi(r\zeta))| ~\leq~ \int\limits_{\sphere} P(\varphi(r\zeta),\eta)|f^{\ast}(\eta)|  d\sigma(\eta).
\]
By Jensen's inequality,
\begin{equation}\label{eqn:jensen}
|f(\varphi(r\zeta))|^p ~\leq~ \int\limits_{\sphere}|f^{\ast}(\eta)|^p   P(\varphi(r\zeta),\eta) d\sigma(\eta),
\end{equation}
and hence
\begin{align*}
\int\limits_{\sphere} & \left|f(\varphi(r\zeta)) \, |D\varphi(r\zeta)|^{\frac {n-2}{2}}\right|^p d\sigma(\zeta)\notag \\
& \quad ~\leq~ \int\limits_{\sphere} |f^{\ast}(\eta)|^p
\bigg\{ \int\limits_{\sphere} P(\varphi(r\zeta),\eta) \, |D\varphi(r\zeta)|^{\frac {p(n-2)}{2}} d\sigma(\zeta) \bigg\}d\sigma(\eta)\\
& \quad ~\leq~ \big \| |D\varphi|^{\frac {n-2}{2}}\big\|_{\infty}^{p-1} \int\limits_{\sphere} |f^{\ast}(\eta)|^p
\bigg\{ \int\limits_{\sphere} P(\varphi(r\zeta),\eta) \, |D\varphi(r\zeta)|^{\frac {n-2}{2}} d\sigma(\zeta) \bigg\}d\sigma(\eta)\\
& \quad ~\leq~ \bigg\{ \frac {1+|\varphi(0)|} {(1-|\varphi(0)|)^{n-1}}  |D\varphi(0)|^{\frac {n-2}{2}} \bigg\}
\big \| |D\varphi|^{\frac {n-2}{2}}\big\|_{\infty}^{p-1} \, \| f^{\ast}\|_{L^p(\sphere)}^p.
\end{align*}
Again, in the last line we used \eqref{eqn:Poisphi}. This implies
\begin{equation}\label{eqn:upperest2}
\|W_{\varphi,\psi}\|_{h^p\to h^p} ~\leq~  \left\{\frac {1+|\varphi(0)|} {(1-|\varphi(0)|)^{n-1}}
|\psi(0)| \right\}^{\frac {1}{p}} \| \psi \|_{\infty}^{1-\frac {1}{p}}.
\end{equation}

\subsubsection*{Case III: $p=\infty$}  Trivial.

%\end{proof}

\section{Proof of Theorem \ref{thm:main3}}

%\subsection{Some auxiliary lemmas}

%\begin{corollary}
%Suppose $1<p<\infty$, we have
%\[
%
%\]
%\end{corollary}

%The proof of theorem \ref{thm:main2} will be finished by the following two parts.
%
%\textit{Part I:}
%\begin{proof}
\subsection{The upper estimate.}

When $p=1$, the upper estimate
\[
\|W_{\varphi,\psi} \|_{h^1\to h^1} ~\leq~ \bigg(\frac{1+|\varphi(0)|}{1-|\varphi(0)|}\bigg)^{\frac {n}{2}}
\]
follows immediately from \eqref{eqn:upperest1}, together with the observation that
\[
\psi(0) ~=~ |D\varphi(0)|^{\frac {n-2}{2}} ~=~ \big( 1-|\varphi(0)|^2 \big)^{\frac {n-2}{2}}
\]
for any $\varphi\in \mob(\ball)$.

When $p=\infty$, the upper estimate reads
\[
\|W_{\varphi,\psi} \|_{h^\infty\to h^\infty} ~\leq~ \bigg(\frac{1+|\varphi(0)|}{1-|\varphi(0)|}\bigg)^{\frac {n-2}{2}},
\]
which is immediate, since
\[
\psi(x) ~=~ |D\varphi(x)|^{\frac {n-2}{2}} ~\leq~ \bigg( \frac {1+|\varphi(0)|} {1-|\varphi(0)|} \bigg)^{\frac {n-2}{2}}
\]
for any $\varphi\in \mob(\ball)$ and any $x\in \ball$, in view of \eqref{eqn:jac-est}.

Now we let $1<p<\infty$ and $f\in h^p(\ball)$. By Lemma \ref{lem: aast}, we see that
\begin{align*}
\|W_{\varphi,\psi} f\|_{h^p}^p
%=~&\int\limits_{\sphere}|(f\circ\varphi)^\ast(\zeta)|^p\cdot|D\varphi(\zeta)|^{\frac {np}{2}-p} d\sigma(\zeta)\\
%%=~&\int\limits_{\sphere}|f^\ast\circ\varphi_a^\ast|^p\cdot\bigg(\frac{1-|a|^2}{[\zeta, a]^2}\bigg)^{\frac{np}2-p} d\sigma(\zeta)\\
~=~\int\limits_{\sphere} |f^\ast(\varphi(\zeta))|^p |D\varphi(\zeta)|^{\frac {(n-2)p}{2}}  d\sigma(\zeta).
\end{align*}
Making the change of variables  $\eta=\varphi(\zeta)$ ,  using Lemma \ref{lem:Jacobiformula} and \eqref{eqn:jac-est},
we obtain
\begin{align*}
\|W_{\varphi,\psi} f\|_{h^p} ~=~& \bigg\{ \int\limits_{\sphere}|f^\ast(\eta)|^p
\big| D\varphi^{-1}(\eta) \big|^{(n-1)-\frac {p(n-2)}{2}}d\sigma(\eta) \bigg\}^{\frac {1}{p}} \\
\leq~& \left(\dfrac {1+|\varphi(0)|}{1-|\varphi(0)|} \right)^{\left|\frac {n-1}{p}-\frac {n-2}{2}\right|}
\|f^{\ast}\|_{L^p(\sphere)}\\
=~& \left(\dfrac {1+|\varphi(0)|}{1-|\varphi(0)|} \right)^{\left|\frac {n-1}{p}-\frac {n-2}{2}\right|}
\|f\|_{h^p}.
\end{align*}
%where the last inequality follows from the easy estimate
%\[
%\frac {1-|\varphi(0)|}{1+|\varphi(0)|} ~\leq~ |D\varphi^{-1}(\eta)| ~\leq~ \frac {1+|\varphi(0)|}{1-|\varphi(0)|}.
%\]

\subsection{The lower estimate.}

Let $\varphi\in \mob(\ball)$ and $\psi=|D\varphi|^{\frac {n-2}{2}}$.

First, it is obvious that
\[
\|W_{\varphi,\psi}\|_{h^\infty\to h^\infty} ~\geq~ \|\psi\|_{\infty}
~\geq~ \bigg(\frac{1+|\varphi(0)|}{1-|\varphi(0)|}\bigg)^{\frac {n-2}{2}}.
\]

Now we assume that $1\leq p<\infty$. We denote by $W_{\varphi,\psi}^{\star}$ the adjoint of $W_{\varphi,\psi}$ on the dual space
$h^{p^{\prime}}$ of $h^p(\ball)$, where $p^{\prime}$ is the conjugate exponent of $p$.

\begin{lemma}\label{lem:adjoint}
%Let $\varphi\in \mob(\ball)$ and $\psi=|D\varphi|^{\frac {n-2}{2}}$.
%Then
$W_{\varphi,\psi}^{\star} P_y = \psi(y) P_{\varphi(y)}$
for every $y\in \ball$.
\end{lemma}

\begin{proof}
Let $\langle \cdot, \cdot\rangle$ denote the duality pairing between $h^p(\ball)$ and $h^{p^{\prime}}(\ball)$.
More precisely, for $f\in h^p(\ball)$ and $g\in h^{p^{\prime}}(\ball)$,
\[
\langle f, g\rangle=\int\limits_{\sphere} f^{\ast}(\zeta) g^{\ast}(\zeta) d\sigma(\zeta),
\]
with $f^{\ast}$ and $g^{\ast}$ the boundary functions of $f$ and $g$ respectively.
Let $y\in \ball$ be fixed and let $f$ be an arbitrary function in $h^p(\ball)$.
Note that $(P_y)^{\ast}=P(y, \cdot)$ and hence
\[
\langle f, P_y\rangle = \int\limits_{\sphere} f^{\ast}(\zeta) P(y, \zeta) d\sigma(\zeta) = f(y).
\]
Therefore,
\begin{align*}
\langle f, W_{\varphi,\psi}^{\star} P_y\rangle = \langle \psi\cdot (f\circ \varphi),  P_y\rangle = \psi(y) f(\varphi(y))
= \langle f, \psi(y) P_{\varphi(y)}\rangle
\end{align*}
and the lemma follows.
\end{proof}

\begin{remark}
A similar idea gives an explicit formula for the adjoint $W_{\varphi,\psi}^{\star}$ of $W_{\varphi,\psi}$:
\begin{equation}\label{eqn:adj}
(W_{\varphi,\psi}^{\star} f) (y) ~=~ \int\limits_{\sphere} f^{\ast}(\zeta) \left\{\frac {1-|\varphi^{-1}(0)|^2} {|\varphi^{-1}(0)-\zeta|^2}\right\}^{\frac {n-2}{2}}
\frac {1-|y|^2} {|y -\varphi(\zeta)|^n} d\sigma(\zeta).
\end{equation}
Indeed,
\begin{align*}
(W_{\varphi,\psi}^{\star} f) (y) ~=~& \langle W_{\varphi,\psi}^{\star} f, P_y\rangle ~=~ \langle f,  W_{\varphi,\psi} P_y\rangle \\
%~=~& \langle f, \psi \cdot (P_y \circ \varphi) \rangle\\
=~& \int\limits_{\sphere} f^{\ast}(\zeta) \psi(\zeta) P(y, \varphi(\zeta)) d\sigma(\zeta),
\end{align*}
proving \eqref{eqn:adj}.
\end{remark}

%First, it is obvious that
%\[
%\|W_{\varphi,\psi}\|_{h^\infty\to h^\infty} ~\geq~ \|\psi\|_{\infty}
%~\geq~ \bigg(\frac{1+|\varphi(0)|}{1-|\varphi(0)|}\bigg)^{\frac {n-2}{2}}.
%\]
%
%Now we assume that $1\leq p<\infty$.
Now we return to the proof of Theorem \ref{thm:main3}.
It follows from Lemmas \ref{lem:adjoint} and \ref{lem:thenormofposs} that
\begin{align}\label{eqn:WaPoisson}
\frac{\|W_{\varphi,\psi}^{\star} P_y\|_{h^{p^\prime}}} {\|P_y\|_{h^{p^\prime}}}
~=~& \frac {\|P_{\varphi(y)}\|_{h^{p^{\prime}}}\psi(y)}  {\|P_y\|_{h^{p^\prime}}}\\
~=~& \frac {\Phi_{p^\prime}(|\varphi(y)|^2)\, (1-|\varphi(y)|^2)^{\frac {1-n}{p}}}
{\Phi_{p^\prime}(|y|^2)\, (1-|y|^2)^{\frac {1-n}{p}}}
\, |D\varphi(y)|^{\frac {n-2}{2}} \notag \\
~=~& \frac {\Phi_{p^\prime}(|\varphi(y)|^2)} {\Phi_{p^\prime}(|y|^2)}
\, |D\varphi(y)|^{\frac {n-2}{2}-\frac {n-1}{p}}, \notag
\end{align}
where in the last equality we used the formula \eqref{eqn:moebiden2}.
Write $b:= \varphi^{-1}(0)/|\varphi^{-1}(0)|$. Then,
by \eqref{eq:limitpoisson},
\begin{equation}
\lim_{y\to b} \frac {\Phi_{p^\prime}(|\varphi(y)|^2)} {\Phi_{p^\prime}(|y|^2)} ~=~
\lim_{y\to -b} \frac {\Phi_{p^\prime}(|\varphi(y)|^2)} {\Phi_{p^\prime}(|y|^2)} ~=~ 1,
\end{equation}
since $|\varphi(y)|\to 1$ as $|y|\to 1$. %and
%\begin{align*}
%\frac{\|W_{\varphi_a}^\ast P_y\|_{h^\infty}}{\|P_y\|_{h^\infty}}~&=\frac{(1-|y|^2)^{n-1}E_\infty(\varphi_a(y))}
%{(1-|\varphi_a(y)|^2)^{n-1}E_\infty(y)}|D\varphi_a(y)|^{\frac {n-2}{2}}\\
%~&=\frac{E_\infty(\varphi_a(y))}{E_\infty(y)}|D\varphi_a(y)|^{\frac {n}{2}}
%\end{align*}
%with $E_\infty(y):=(1-|y|^2)^{n-1}\|P_y\|_{h^\infty}$.
%Also, recalling that
%\[
%|D\varphi(y)| ~=~\frac{1-|\varphi(0)|^2}{[y, \varphi^{-1}(0)]^2},
%\]
Also, it is clear that
\begin{align}\label{eq:limitofvarphia}
\lim_{y\to b} |D\varphi(y)| ~=~ \frac {1+|\varphi(0)|}{1-|\varphi(0)|} \\
\intertext{and}
\lim_{y\to -b} |D\varphi(y)| ~=~ \frac {1-|\varphi(0)|}{1+|\varphi(0)|}.
\end{align}
Therefore,
\begin{equation}\label{eq:normreproducing}
\limsup_{|y|\to 1^{-}}\frac{\|W_{\varphi,\psi}^{\star} P_y\|_{h^{p^\prime}}}{\|P_y\|_{h^{p^\prime}}}
~\geq~  \left(\dfrac {1+|\varphi(0)|}{1-|\varphi(0)|} \right)^{\left|\frac {n-1}{p}-\frac {n-2}{2}\right|}.
\end{equation}
It follows that
\[
\|W_{\varphi,\psi}\|_{h^p\to h^p} ~=~ \|W_{\varphi,\psi}^{\star}\|_{h^{p^\prime}\to h^{p^\prime}}
~\geq~  \left(\dfrac {1+|\varphi(0)|}{1-|\varphi(0)|} \right)^{\left|\frac {n-1}{p}-\frac {n-2}{2}\right|}
\]
as desired.
%\end{proof}

%\subsubsection*{Case II: $p=\infty$}
%It is clear that
%\[
%\|W_ \varphi\|_{h^\infty\to h^\infty} ~\geq~ \|W_ \varphi \mathbf{1}\|_{h^\infty}
%~\geq~ \bigg(\frac{1+|\varphi(0)|}{1-|\varphi(0)|}\bigg)^{\frac {n-2}{2}}.
%\]

\section{Proof of theorem \ref{thm:essentialnorm}}

\begin{lemma}\label{lem:weakly}
Let $k_y^{(p)} := P_y/\|P_y\|_{h^p}$. Suppose $1<p<\infty$. Then $k_y^{(p)} \to 0$ weakly in $h^p(\ball)$ as $|y|\to 1$.
\end{lemma}
\begin{proof}
Clearly, the family $\{k_y^{(p)}\}_{y\in \ball}$ is bounded in $h^p(\ball)$.  Also, it is immediate from \eqref{eq:normpoisson} that
$k_y^{(p)}\to 0$ uniformly on every compact subset of $\ball$ as $|y|\to 1$.
In view of \cite[Corollary 1.3]{CM95},  this implies $k_y^{(p)} \to 0$ weakly in $h^p(\ball)$.
\end{proof}

We proceed to the proof of theorem \ref{thm:essentialnorm}.
%\begin{proof}
Since $\|W_{\varphi, \psi}\|_{e} \leq \|W_{\varphi, \psi}\|_{h^p\to h^p}$, it follows from Theorem \ref{thm:main3} that
\[
\|W_{\varphi, \psi}\|_{e} ~\leq~ \left(\dfrac {1+|\varphi(0)|}{1-|\varphi(0)|} \right)^{\left|\frac {n-1}{p}-\frac {n-2}{2}\right|}.
\]
To prove the converse inequality, let $\{y_j\}_{j=1}^{\infty}$ be in $\ball$ tending to $\sphere$.
If $K$ is an arbitrary compact operator on $h^p(\ball)$,
we have $K^\star (k_{y_j}^{(p^{\prime})}) \to 0$, by Lemma \ref{lem:weakly}.
Since $K$ is compact,
\begin{align*}
\|W_{\varphi, \psi}-K\|_{h^p\to h^p} ~=~& \|(W_{\varphi, \psi}-K)^\star\|_{h^{p^\prime}\to h^{p^\prime}}\\
\geq&~ \limsup_{j\to \infty}\|(W_{\varphi, \psi}-K)^\star(k_{y_j}^{(p^{\prime})}) \|_{h^{p^\prime}}\\
\geq &~ \limsup_{j\to \infty}\|W_{\varphi, \psi}^\star(k_{y_j}^{(p^{\prime})})\|_{h^{p^\prime}}.
\end{align*}
Thus,
\[
\|W_{\varphi, \psi}-K\|_{h^p\to h^p} \geq  \limsup_{|y|\to 1^{-}}\|W_{\varphi,\psi}^\star(k_{y}^{(p^{\prime})})\|_{h^{p^\prime}} .
\]
Now, an application of \eqref{eq:normreproducing} completes the proof.

%-------------------------------------------------------------------

\end{document}